\documentclass[english]{article}
\pdfoutput=1

\usepackage[citecolor=blue]{hyperref} 
\usepackage[T1]{fontenc}
\usepackage[latin9]{inputenc}
\usepackage{mathrsfs}
\usepackage{amsthm}
\usepackage{amsmath}
\usepackage{amssymb}
\usepackage{graphicx}
\newcommand{\CA}{\mbox{$\mathsf{CA}$}}
\newcommand{\CAN}{\mbox{$\mathsf{CAN}$}}

\makeatletter

\theoremstyle{plain}
\newtheorem{thm}{\protect\theoremname}
  \theoremstyle{plain}
  \newtheorem{lem}[thm]{\protect\lemmaname}

\usepackage{fullpage}

\makeatother

\usepackage{babel}
  \providecommand{\lemmaname}{Lemma}
\providecommand{\theoremname}{Theorem}

\begin{document}

\title{Upper bounds on the size of covering arrays }

\author{Kaushik Sarkar and Charles J. Colbourn\\
School of Computing, Informatics, and Decision Systems Engineering\\
Arizona State University, PO Box 878809\\
Tempe, Arizona, 85287-8809, U.S.A.\\
ksarkar1, colbourn@asu.edu}

\maketitle
\begin{abstract}
Covering arrays find important application in software and hardware interaction testing. 
For practical applications it is useful to determine or bound the minimum number of rows, $\CAN(t,k,v)$, in a covering array for given values of the parameters $t,k$ and $v$. 
Asymptotic upper bounds for $\CAN(t,k,v)$ have earlier been established using the Stein-Lovász-Johnson strategy and the Lovász local lemma. 
A series of improvements on these bounds is developed in this paper.  
First an estimate for the discrete Stein-Lovász-Johnson bound is derived. 
Then using alteration,  the Stein-Lovász-Johnson bound is improved upon, leading to a two-stage  construction algorithm. 
Bounds from the Lovász local lemma are improved upon in a different manner, by examining group actions on the set of symbols.
Two asymptotic upper bounds on $\CAN(t,k,v)$ are established that are tighter than the known bounds. 
A two-stage  bound is derived that employs the Lovász local lemma and the conditional Lovász local lemma distribution. 
\end{abstract}

\section{Introduction}

Let $N,\, t,\, k,$ and $v$ be positive integers with $k\ge t\ge2$ and $v\ge2$. 
A \emph{covering array} $\CA(N;t,k,v)$ is an $N\times k$ array $A$ in which each entry is from a $v$-ary alphabet $\Sigma$, and for every $N\times t$ sub-array $B$ of $A$ and every  ${\bf x} \in \Sigma^{t},$ there is a  row of $B$ that equals $\bf x$. 

When $k$ is a  positive integer,  $[k]$  denotes the set $\{1,\ldots,k\}$. 
A \emph{$t$-way interaction}  is $\{(c_{i},a_{i})\,:\,1\le i\le t,\, c_{i}\in[k],\, c_{i}\neq c_{j}\,\text{for }i\neq j,\,\text{and }a_{i}\in\Sigma\}$.
So an interaction is an assignment of values from  $\Sigma$ to  $t$  of the $k$  columns. 
An $N\times k$ array $A$ \emph{covers} the interaction $\iota = \{(c_{i},a_{i})\,:\,1\le i\le t,\, c_{i}\in[k],\, c_{i}\neq c_{j}\,\text{for }i\neq j,\,\text{and }a_{i}\in\Sigma\}$ if there is a row $r$ in $A$ such that $A(r,c_{i})=a_{i}$ for $1\le i\le t$. When there is no such row in
$A$,  $\iota$ is \emph{not} covered in $A$. 
Hence a $\CA(N;t,k,v)$ covers all the $t$-way interactions involving $k$ columns each having $v$ values. 

Covering arrays are used extensively for interaction testing in complex engineered systems. 
In that setting, the $k$ columns represent {\em factors} that may affect performance; the $v$ values are the valid {\em levels} of the factors; each of the $N$ rows forms a {\em test} of a test suite; and $t$ is the {\em strength} of coverage of interactions among the factors.
Real-world software or hardware system can consist of hundreds of components. 
While unit testing can reveal faulty selections for particular components, correct components may nevertheless interact to cause a fault.
To ensure that all  possible combinations of options of $t$ components function together correctly, one needs to examine all possible $t$-way interactions. 
When the number of components is $k$, and the number of different options available for each component is  $v$,  the $N$ tests  of a $\CA(N;t,k,v)$ collectively test all  $t$-way interactions. 
For this reason, covering arrays have been used in combinatorial interaction testing in varied fields like software and hardware engineering, design of composite materials, and biological networks \cite{CAWSE,Kuhnbook,KUHN2,RonC,SerB}. 

The cost of testing is directly related to the number of test cases, so one is interested in covering arrays with the fewest rows.
The smallest value of $N$ for which $\CA(N;t,k,v)$ exists is denoted by $\CAN(t,k,v)$. 
Efforts to determine or bound $\CAN(t,k,v)$ have been extensive; see \cite{sicily,croatia,Kuhnbook,NieL-CS} for example. 
Naturally one would prefer to determine $\CAN(t,k,v)$ exactly.  
Katona  and Kleitman \cite{katona}   and Spencer \cite{kleitman} independently showed that for $t=v=2$,
the minimum number of rows $N$ in a $\CA(N;2,k,2)$ is the smallest $N$ for which 
$k\leq{N-1 \choose \lceil\frac{N}{2}\rceil}$.
Since that time, the exact value of $\CAN(t,k,v)$ as a function of $k$ has not been determined for any other cases with $t \geq 2$ and $v \geq 2$.

In light of this, the asymptotic determination of $\CAN(t,k,v)$ has been of substantial interest.  
For fixed $t$ and $v$, it is well-known that $\CAN(t,k,v)$ is $\Theta(\log k)$; the lower bound can be established by observing that all columns are distinct, and the upper bound is a simple probabilistic argument.
When $t=2$ and  $v\ge2$. Gargano et al. \cite{GKV2} establish the much more precise statement that  $\CAN(t,k,v)=\frac{v}{2}\log k\left\{ 1+o(1)\right\} $. 
However, when $t > 2$, even the coefficient of the highest order term is not known precisely.
One of our main results improves on the  best known asymptotic upper bound on $\CAN(t,k,v)$. 

Returning to the testing application, the methods used to obtain asymptotic bounds have had little impact to date, for two main reasons. 
First, other methods typically provide smaller arrays than are guaranteed by the asymptotic methods.  
Secondly, even when the asymptotic methods yield a better bound, it may be non-constructive or provide no efficient construction method.  
To understand these, consider the current tables of upper bounds for covering array numbers \cite{catables}.
When $t=6$ and $v=3$, for example, direct constructions \cite{WC} determine the best known upper bounds on $\CAN(6,k,3)$ when $k \leq 14$; greedy algorithms \cite{ColCECA,ipogf}  determine bounds for $15 \leq k \leq 51$; and recursive methods \cite{CZquilt} determine bounds for $k \geq 52$. 
Each of these provides an efficient method of producing the covering array for use in testing, yet for $k \geq 53$ the sizes of the arrays so produced are {\sl larger} than one guaranteed to exist by probabilistic arguments.  
Evidently efficient constructions to implement the asymptotic methods show much promise in producing covering arrays for testing large systems.  
Our second main contribution is to demonstrate that the improvements in the asymptotic bound form the basis of efficient construction algorithms.  

Next we introduce some notation used throughout the paper.
Let $\mathscr{I}_{t}$ be the set of all $t$-way interactions on $k$ factors with $v$ levels, and let $\mathscr{C}_{t}$ be the set of all subsets of size $t$ of $[k]$, i.e. $\mathscr{C}_{t}={[k] \choose t}$.
We represent each $t$-subset of $[k]$ as an increasing sequence of $t$ values from $[k]$, so that each $t$-subset has a unique representation. 
We often abuse the notation  and use this sequence representation for the subset. 
Define  $c:\mathscr{I}_{t}\rightarrow\mathscr{C}_{t}$ as follows: For  $\iota\in\mathscr{I}_{t}$, $c(\iota)=(c_{1},\ldots,c_{t})$ where $(c_{i},a_{i})\in\iota$ for some $a_{i}\in\Sigma$, and $c_{i}<c_{j}$ for $1\le i<j\le t$.
We  use $c(\iota)_{i}$ to denote the $i$th element of $c(\iota)$, i.e. $c_{i}$. 
Similarly, define  $s:\mathscr{I}_{t}\rightarrow\Sigma^{t}$ as follows: For an interaction $\iota\in\mathscr{I}_{t}$, define $s(\iota)=(a_{1},\ldots,a_{t})$ where $(c(\iota)_{i},a_{i})\in\iota$
for $1\le i\le t$. 
We  use $s(\iota)_{i}$ to denote the $i$th element of the $t$-tuple $s(\iota)$. 
A bijection between $\mathscr{I}_{t}$ and $\mathscr{C}_{t}\times\Sigma^{t}$ maps $\iota\rightarrow(c(\iota),s(\iota))$. 
Therefore, the interaction $\iota$ can be  described by the ordered pair $(c(\iota),s(\iota))$.

The rest of the paper is organized as follows. 
Section \ref{sec:slj} introduces the Stein-Lovász-Johnson bound on $\CAN(t,k,v)$.
We develop a discrete version of Stein-Lovász-Johnson bound and provide a useful estimate of this bound. 
In Section \ref{sub:Partial-array} we present our first result --- an upper bound on $\CAN(t,k,v)$. 
The statement and proof of the bound are followed by a discussion of its constructive nature. 
Section \ref{sec:Partial-dependence} first discusses the partial dependence structure of the interactions, and derives the Godbole--Skipper--Sunley (GSS) bound. 
Section \ref{sub:Group-action} presents the key improvement on bounds for covering array numbers. 
It applies group actions  to covering arrays  to improve the GSS bound. 
We combine the ideas of Section \ref{sub:Partial-array} and \ref{sub:Group-action} to obtain yet another upper bound on $\CAN(t,k,v)$ in Section \ref{sec:Partial-array-LLL}. 
In Section \ref{sec:Conclusion} we discuss the relative merits of the different bounds ontained, and present  an open problem.

\section{The Stein-Lovász-Johnson bound}\label{sec:slj}

Specializing the method of Stein \cite{Stein74}, Lovász \cite{Lovasz75}, and Johnson \cite{Johnson74} to the case of covering arrays one gets an upper bound on $\CAN(t,k,v)$ in the general case. 
Because the ideas used are essential for the rest of the paper, we provide a  proof of this known result.
\begin{thm}
\label{thm:slj}{\rm \cite{Johnson74,Lovasz75,Stein74}}(Stein-Lovász-Johnson
(SLJ) bound): Let $t,\, k,\, v$ be integers with $k\ge t\ge2$, and
$v\ge2$. Then as $k\rightarrow\infty$,
\[
\CAN(t,k,v)\le\frac{t}{\log\frac{v^{t}}{v^{t}-1}} \log k (1+\mbox{\rm o}(1))
\]
\end{thm}
\begin{proof}
Let $A$ be an $N\times k$ array in which each entry is chosen independently and uniformly at random from an alphabet $\Sigma$ of size $v$.
The probability that a specific interaction of strength $t$ is not  covered in $A$ is $\left(1-\frac{1}{v^{t}}\right)^{N}$. 
By the linearity of expectations, the expected number of uncovered interactions in $A$ is ${k \choose t}v^{t}\left(1-\frac{1}{v^{t}}\right)^{N}$.
If this expectation is less than $1$, because the number of uncovered interactions is an integer,  there is an array with $N$ rows that covers all the interactions.
Solving ${k \choose t}v^{t}\left(1-\frac{1}{v^{t}}\right)^{N}<1$,
we get $\CAN(t,k,v)\le\frac{\log{k \choose t}+t\log v}{\log\left(\frac{v^{t}}{v^{t}-1}\right)}$.
Simplifying further,

\begin{eqnarray*}
\CAN(t,k,v) & \le & \frac{\log{k \choose t}+t \log v}{\log\left(\frac{v^{t}}{v^{t}-1}\right)}\\
 & \le & \frac{t \log\left(\frac{ke}{t}\right)+t \log v}{\log\left(\frac{v^{t}}{v^{t}-1}\right)}\\
 & = & \frac{t \log k}{\log\left(\frac{v^{t}}{v^{t}-1}\right)}\left(1+\frac{1}{\log k}-\frac{\log t}{\log k}+\frac{\log v}{\log k}\right)\\
 & = & \frac{t}{\log\frac{v^{t}}{v^{t}-1}}\log k (1+\mbox{\rm o}(1))
\end{eqnarray*}

This completes the proof.
\end{proof}

Rather than choosing the $N$ rows at random, we can build the covering array one row at a time.
To select a row, compute the expected number of uncovered interactions that remain when we choose the next row uniformly at random from $\Sigma^{k}$.
There must be a row whose selection leaves at most that expected number of interactions uncovered. 
Indeed except when the first row is selected, some row must leave a number that is strictly less than the expectation, because previously selected rows cover no interaction that is not yet covered.
Add such a row to the covering array and repeat  until all the interactions are covered. 
The number of  rows employed by this method yields an upper bound on $\CAN(t,k,v)$. 
If at each stage the row selected left uncovered precisely the expected number of uncovered interactions, we recover Theorem \ref{thm:slj}.  
However, after each row selection the number of uncovered interactions must be an integer no larger than the expected number,  improving on the basic SLJ bound.
The better upper bound is the   \emph{discrete Stein-Lovász-Johnson} (discrete-SLJ) bound. 

A row that leaves no more than the expected number uncovered can be computed efficiently when $t$ and $v$ are fixed, so the discrete-SLJ bound can be efficiently derandomized; this is the basis of the {\em density algorithm} \cite{DDA,nDDA}.
The density algorithm works quite well in practice, providing the smallest known covering arrays in many cases \cite{catables}. 
Although Theorem \ref{thm:slj} provides an easily computed upper bound on the array sizes produced by the density algorithm, it is a very loose  bound.

We analyze the discrete Stein-Lovász-Johnson bound in order to establish a better estimate. 
\begin{thm} 
The number of rows $N$ in $A$ obtained by the discrete-SLJ bound satisfies \[ \frac{\log\left\{ {k \choose t}+1\right\} }{\log\left(\frac{v^{t}}{v^{t}-1}\right)} < N \le \frac{\log\left\{ {k \choose t}+\epsilon\right\} -\log\epsilon}{\log\left(\frac{v^{t}}{v^{t}-1}\right)} \] for some $0 < \epsilon < 1$.
\end{thm}
\begin{proof}
Let $y=\left(1-\frac{1}{v^{t}}\right)$ and $x = 1/y$. 
Let $r(i)$ denote the number of uncovered interactions that remain after $i$ rows are chosen.
Suppose that when row $i$ is chosen, it leaves 
\[ r(i)= \left \{ \begin{array}{lcl} \lfloor y r(i-1)\rfloor & \mbox{when} & i=1 \mbox{ or } r(i-1) \not\equiv 0 \pmod{v^t} \\
y r(i-1)-1 & \mbox{when} & i>1 \mbox{ and } r(i-1) \equiv 0 \pmod{v^t}
\end{array} \right . \]
uncovered interactions.

Write $\epsilon(i-1) = y r(i-1) - r(i)$ for $i \geq 1$.
Then expanding the recurrence $r(i)  = y r(i-1) - \epsilon(i-1)$, 
\[ r(n) = y^n r(0) - \sum_{i=0}^{n-1} y^{n-1-i} \epsilon(i) . \]
Rewriting in terms of $x$, 
\[ x^n r(n) =  r(0) - \sum_{i=0}^{n-1} x^{i+1}  \epsilon(i) . \]
Now $r(0) = {k \choose t}v^{t}$ and $r(n) = 0$, so 
\[ {k \choose t}v^{t} = x^n \epsilon(n-1) +  \sum_{i=0}^{n-2} x^{i+1}  \epsilon(i) . \]
Because $r(n) = 0$, $y \leq \epsilon(n-1) < 1$. Then because $0 \leq \epsilon(i) \leq 1$,   
\[ x^{n-1}  +  \sum_{i=0}^{n-2} x^{i+1}  \epsilon(i) \leq {k \choose t}v^{t} < x^n  +  \sum_{i=0}^{n-2} x^{i+1}  \epsilon(i) <  \sum_{i=1}^{n} x^{i}  = \frac{x(x^{n}-1)}{(x-1)} . \]

Simplify  ${k \choose t}v^{t} <  \frac{x(x^{n}-1)}{(x-1)}$ to obtain ${k \choose t}+1 <  x^{n}$.
Take logarithms to establish that $n>\frac{\log\left\{ {k \choose t}+1\right\} }{\log\left(\frac{v^{t}}{v^{t}-1}\right)}$.
If we select each row so that  $r(n)=\lfloor y r(n-1)\rfloor$,  we cannot cover all interactions in $\log\left\{ {k \choose t}+1\right\} /\log\left(\frac{v^{t}}{v^{t}-1}\right)$ rows. 
This establishes the lower bound.

Note that $\epsilon(0) = 0$.  Let $\epsilon=\min\{\epsilon(k): 1\le k<n-1\}$. 
Then $\frac{1}{v^t} \leq \epsilon \leq 1$, because every row selected after the first covers more than the expected number of previously uncovered interactions.
Then for sufficiently large $k$ 
\[  \begin{array}{rcl} \epsilon \frac{x(x^{n-1}-1)}{(x-1)} &= & \epsilon \sum_{i=1}^{n-1} x^{i}  \leq x^{n-1}  +  \epsilon \sum_{i=2}^{n-2} x^{i} \\
& < & x^{n-1}  + \epsilon \sum_{i=2}^{n-2} x^{i} + \epsilon (x^{n-1} - x) \\
& \leq & x^{n-1}  +  \sum_{i=0}^{n-2} x^{i+1}  \epsilon(i) \leq  {k \choose t}v^{t} \end{array}. \]
The strict inequality follows from the fact that $x>1$. Hence $\epsilon (x^{n}-1) <  {k \choose t}$, so $n < \frac{\log\left\{ {k \choose t}+\epsilon\right\} -\log\epsilon}{\log\left(\frac{v^{t}}{v^{t}-1}\right)}+1$,
and because $n$ is an integer the upper bound follows.
\end{proof}

Consequently $\log\left\{ {k \choose t}+1\right\} /\log\left(\frac{v^{t}}{v^{t}-1}\right)$ can be used to estimate  the discrete Stein-Lovász-Johnson bound.
Figure \ref{fig:Comparison-slj-dslj} compares the estimate to the discrete Stein-Lovász-Johnson bound and the Stein-Lovász-Johnson bound from Theorem \ref{thm:slj} when $t=6$ and $v=3$. 
For a wide range of values of $k$, the reduction in the number of rows is substantial. 

\begin{figure}
\begin{centering}
\includegraphics[scale=0.7]{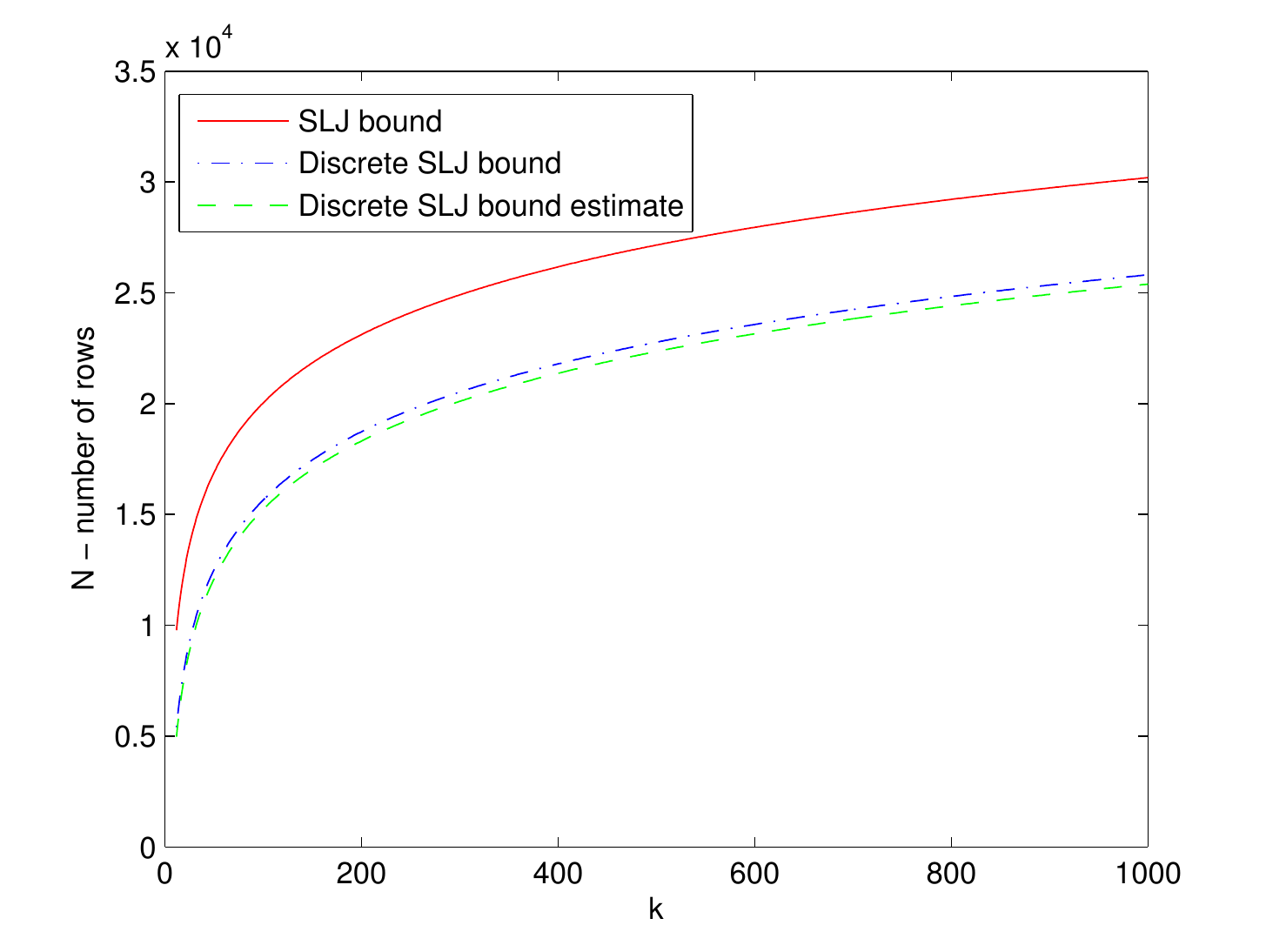}
\par\end{centering}

\caption{Comparison of the Stein-Lovász-Johnson bound, the discrete Stein-Lovász-Johnson
bound, and the estimate for the discrete Stein-Lovász-Johnson
bound. $t=6$, $v=3$.\label{fig:Comparison-slj-dslj} }
\end{figure}

The  density algorithm \cite{DDA,nDDA} enables one to produce covering arrays of sizes at most those given by the bound efficiently.  
Despite their efficiency in theory, in practice the methods are limited by the need to store information about all $t$-way interactions; even when $t=6$, $v=3$, and $k=54$, there are 18,828,003,285 6-way interactions, so the storage requirements are limiting.
Moreover, as shown in the analysis, rows added towards the end of the process account for relatively few of the interactions.  
For these reasons, we explore a two-stage approach using alteration.

\subsection{Constructing and completing a partial array}\label{sub:Partial-array}

Alteration is an important strategy in probabilistic methods \cite{alon08}.
The idea is to consider ``random'' structures that have a few ``blemishes'', in that they do not have all the desired properties. 
Such ``partial'' structures are then altered to obtain the desired property. 
To apply this technique to covering arrays, in stage 1 we construct a random $n\times k$ array with each entry  chosen from the $v$-ary alphabet $\Sigma$ independently and uniformly at random. 
The number of uncovered interactions after stage 1 can be computed using the SLJ or discrete-SLJ bounds.
In stage 2,  we add one new row for each uncovered interaction to obtain a covering array. 

For example, when $t=6,\, k=54$ and $v=3$, Theorem \ref{thm:slj} gives $\CAN(6,54,3) \leq 17,236$.
Using the alteration approach, Figure \ref{fig:minima} plots an upper bound on the size of the completed covering array against the number $n$ of rows in a partial array that covers at least the expected number of interactions.
The smallest covering array is obtained when $n=12,402$, which when completed yields $\CAN(6,54,3) \leq 13,162$. 
At least in this case, our alteration  provides a much better bound. 
We explore this in general. 

\begin{figure}
\begin{centering}
\includegraphics[scale=0.7]{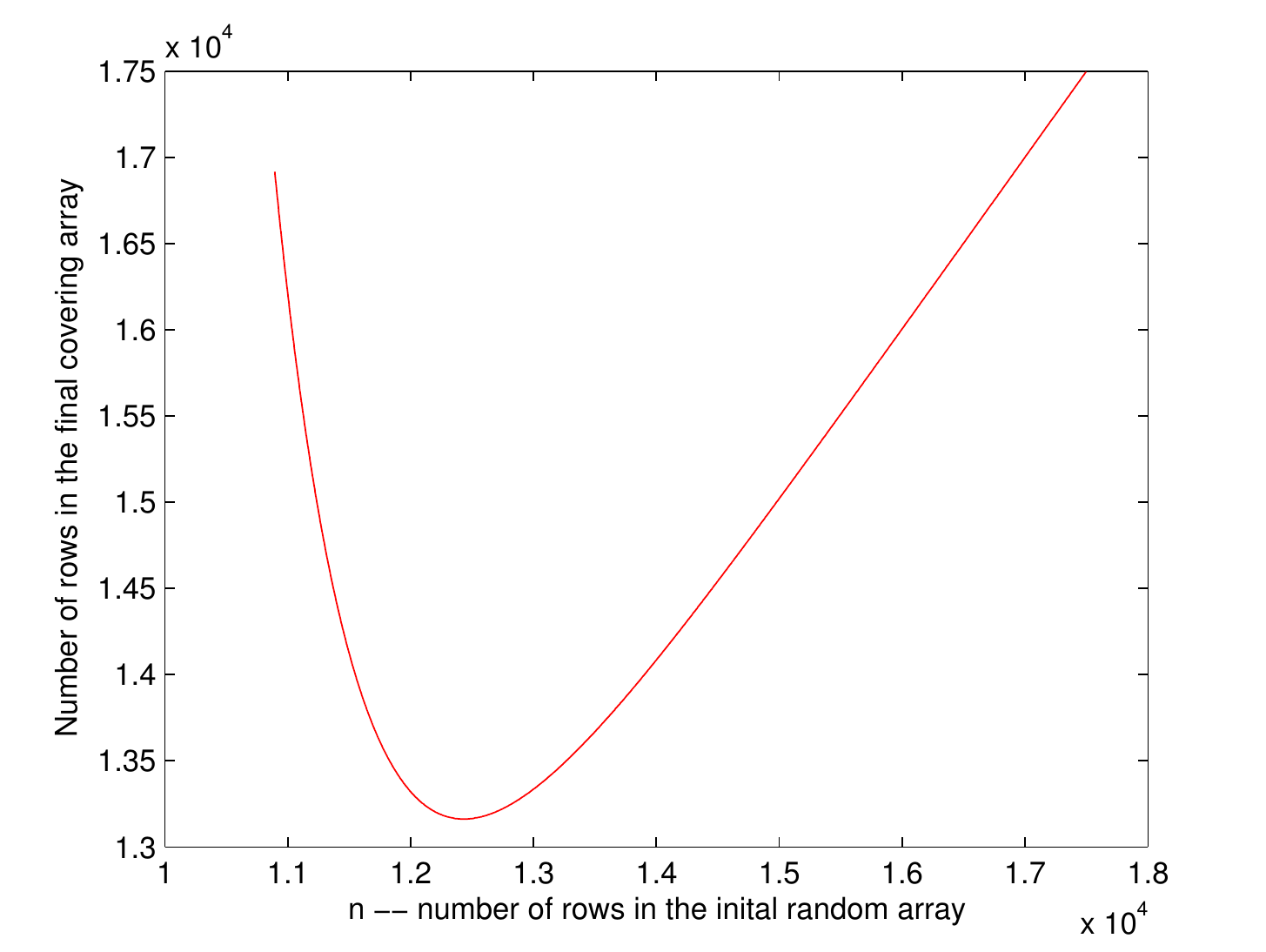}
\par\end{centering}

\caption{Plot of $n+\left \lfloor{k \choose t}v^{t}\left(1-\frac{1}{v^{t}}\right)^{n} \right \rfloor$
against $n$, the size of the partial covering array, for $t=6,\, k=54$,
and $v=3$. ${k \choose t}v^{t}\left(1-\frac{1}{v^{t}}\right)^{n}$
is the expected number of uncovered interactions in a random $n\times k$
array. The minimum number of  rows in the final covering array
is $13,162$,  achieved when the initial random array has $n=12,402$
rows. The Stein-Lovász-Johnson bound requires $17,236$ rows, and
the best known covering array has $17,197$ rows.\label{fig:minima}}
\end{figure}

\begin{thm}
\label{thm:two-stage}Let $t,\, k,\, v$ be integers with $k\ge t\ge2$,
and $v\ge2$. Then \[ \CAN(t,k,v)\le\frac{\log{k \choose t}+t\log v+\log\log\left(\frac{v^{t}}{v^{t}-1}\right)+1}{\log\left(\frac{v^{t}}{v^{t}-1}\right)}.\] \end{thm}
\begin{proof}
In an $n\times k$ array with each entry  chosen independently and uniformly at random from an alphabet $\Sigma$ of size $v$,  the expected number of uncovered $t$-way interactions is ${k \choose t}v^{t}\left(1-\frac{1}{v^{t}}\right)^{n}$.
Let $P$ be  an $n\times k$ array with at most $\lfloor{k \choose t}v^{t}\left(1-\frac{1}{v^{t}}\right)^{n}\rfloor$ uncovered interactions. 
Let $Q$ contain $\lfloor{k \choose t}v^{t}\left(1-\frac{1}{v^{t}}\right)^{n}\rfloor$ new rows, each covering a different interaction not covered in $P$.
Then $A = \left ( {P \atop Q} \right ) $  is a covering array with $n+\lfloor{k \choose t}v^{t}\left(1-\frac{1}{v^{t}}\right)^{n}\rfloor$ rows. 
So an upper bound on the number of rows in $A$ is  $n+{k \choose t}v^{t}\left(1-\frac{1}{v^{t}}\right)^{n}$. 
Applying elementary calculus, the fewest rows is \[ \frac{\log{k \choose t}+t\log v+\log\log\left(\frac{v^{t}}{v^{t}-1}\right)+1}{\log\left(\frac{v^{t}}{v^{t}-1}\right)}, \] obtained when $P$ has $n=\frac{\log{k \choose t}+t\log v+\log\log\left(\frac{v^{t}}{v^{t}-1}\right)}{\log\left(\frac{v^{t}}{v^{t}-1}\right)}$ rows. 
\end{proof}

For $v,t\ge2$, $\log\log\left(\frac{v^{t}}{v^{t}-1}\right)<0$. 
Hence,  Theorem \ref{thm:two-stage} gives a tighter bound on $\CAN(t,k,v)$ than that of Theorem \ref{thm:slj}.
Using the Taylor series expansion of $\log(1-x)$, it can be shown that $1/\log\left(\frac{v^{t}}{v^{t}-1}\right)\le v^{t}$.
In fact, in the range of values of $v$ and $t$ of interest, $1/\log\left(\frac{v^{t}}{v^{t}-1}\right)\approx v^{t}$.
So Theorem \ref{thm:two-stage} guarantees the existence of a covering array with $N\approx\frac{\log{k \choose t}+1}{\log\left(\frac{v^{t}}{v^{t}-1}\right)}\approx v^{t}\log{k \choose t}+v^{t}$ rows. 

The argument in the proof of Theorem \ref{thm:two-stage} can be made constructive.
It underlies an efficient randomized construction algorithm for covering arrays: 
In the first stage,  construct a random $n\times k$ array with $n\approx v^{t}\log{k \choose t}$ rows; then  check if the number of uncovered interactions is at most  $v^{t}$. 
If not,  randomly generate another $n\times k$ array and repeat the check. 
In the second stage  add at most $v^{t}$ rows to the partial covering array to cover the remaining interactions. 
Neither stage needs to store information about individual interactions, because we need only count the uncovered interactions in the first stage.
The second stage is deterministic and efficient.
The first stage has expected polynomial running time; it could be efficiently derandomized in principle using the methods in \cite{DDA,nDDA}, at the price of the storage of the status of individual interactions.

The proof of Theorem \ref{thm:two-stage} suggests a   general ``two-stage''  construction paradigm, in which the first stage uses one strategy to cover almost all of the interactions, and the second uses another to cover the relatively few that remain.
In related work we explore such two-phase methods for the explicit construction of covering arrays \cite{SarkarColbourn2}.

Figure \ref{fig:Comparison-slj-2-stage} compares the two-stage based bound with the Stein-Lovász-Johnson bound and the discrete Stein-Lovász-Johnson bound. 
In the cases shown, the two-stage  bound is much better than the Stein-Lovász-Johnson bound, and  not much worse than the discrete Stein-Lovász-Johnson bound. 
Therefore a purely randomized method (with much smaller memory requirements) produces covering arrays that are competitive with the guarantees from the density algorithm.  

\begin{figure}
\begin{centering}
\includegraphics[scale=0.7]{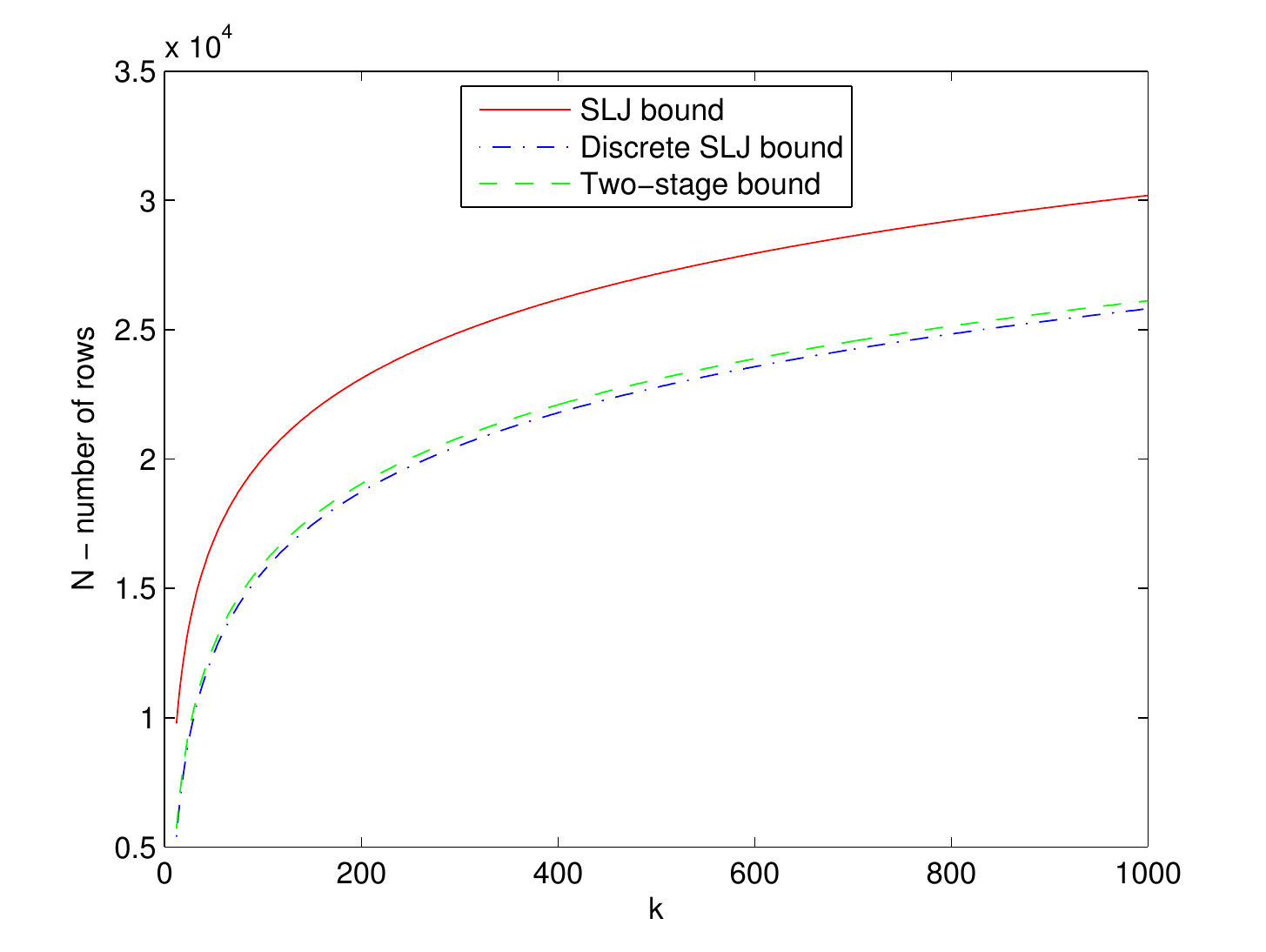}
\par\end{centering}

\caption{Comparison of Stein-Lovász-Johnson bound, discrete Stein-Lovász-Johnson
bound and two-stage based bound from Theorem \ref{thm:two-stage}. $t=6$, $v=3$.\label{fig:Comparison-slj-2-stage} }
\end{figure}

\section{Limited dependence and the Lovász local lemma}\label{sec:Partial-dependence}

When $k\ge2t$, some interactions  have no columns in common. 
The events of coverage of such interactions are independent.
Neither Theorem \ref{thm:slj} nor Theorem \ref{thm:two-stage} takes advantage of this. 
Consider an $N\times k$ array $A$ with each entry  chosen independently and uniformly at random from $\Sigma$. 
Let $A_{\iota}$ denote the event that the interaction $\iota\in\mathscr{I}_{t}$ is not covered in $A$. 
$A_{\iota}$ depends on all  events $\{ A_{\rho} : \rho\in\mathscr{I}_{t}, c(\iota) \cap c(\rho) \neq \emptyset\}$, and only on those events. 
Hence when $k\ge2t$, there are events of which $A_{\iota}$ is  independent. 
Because of this limited dependence, the upper bound on $\CAN(t,k,v)$ from Theorem \ref{thm:slj} can be considerably improved by applying the Lovász local lemma. 

\begin{lem}\label{lllsym}
(Lovász local lemma; Symmetric case) (see {\rm \cite{alon08}}) 
Let $A_{1},A_{2},\ldots,A_{n}$ events in an arbitrary probability space. 
Suppose that each event $A_{i}$ is mutually independent of a set of all other events $A_{j}$ except for at most $d$, and that $\Pr[A_{i}]\le p$ for all $1\le i\le n$.
If $ep(d+1)\le1$, then $\Pr[\cap_{i=1}^{n}\bar{A_{i}}]>0$.
\end{lem}
The symmetric version of Lovász local lemma provides an upper bound on the probability of a ``bad'' event in terms of the dependence structure among such bad events, so that there is a guaranteed outcome in which all  ``bad'' events are avoided. 
In this and following sections we successively improve the upper bound on $\CAN(t,k,v)$ asymptotically by exploiting this limited dependence among interactions through the Lovász local lemma.

To simplify the comparisons, define $d(t,v)=\lim\sup_{k\rightarrow\infty}\frac{\CAN(t,k,v)}{\log k}$.
Theorem \ref{thm:slj} establishes that  $d(t,v)\le\frac{t}{\log\frac{v^{t}}{v^{t}-1}}$.
Using Lemma \ref{lllsym}, Godbole, Skipper, and Sunley \cite{GSS} establish a tighter bound:
\begin{thm}
\label{thm:godbole}{\rm \cite{GSS}} Let $t,\, v$ be integers with $t,v\ge2$. Then
\[
d(t,v)\le\frac{t-1}{\log\frac{v^{t}}{v^{t}-1}}
\]
\end{thm}
Figure \ref{fig:Comparison-slj-lll} compares the bounds from Theorems \ref{thm:slj} and \ref{thm:godbole} for $t=6$ and $v=3$.
The bounds are plotted in log-log scale to highlight the asymptotic difference between them.

\begin{figure}
\begin{centering}
\includegraphics[scale=0.7]{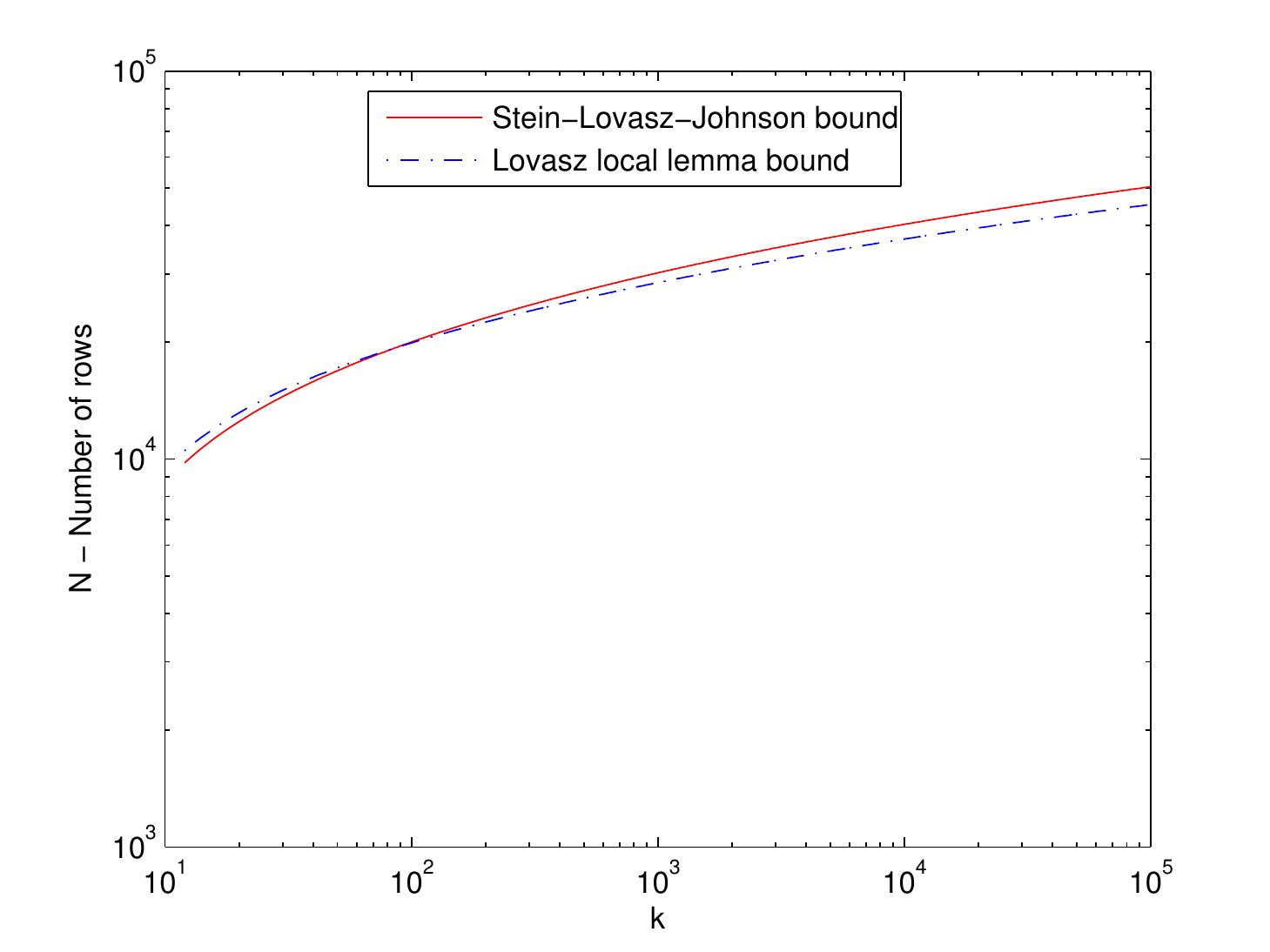}
\par\end{centering}

\caption{Comparison of Stein-Lovász-Johnson and Lovász local lemma bounds for $t=6$ and $v=3$. The graph is plotted in log-log scale. \label{fig:Comparison-slj-lll}}
\end{figure}

\subsection{Group action}\label{sub:Group-action}

We apply the Lovász local lemma in conjunction with a group action. Let $\Gamma$ be a permutation group on the $v$-symbol alphabet $\Sigma$. We define the action of $ \Gamma $ on the set of all $ t $-way interaction on $ k $ factors in the natural way: For $ \sigma \in \Gamma $ and $\iota = \{(c_{i},a_{i})\,:\,1\le i\le t,\, c_{i}\in[k],\, c_{i}\neq c_{j}\,\text{for }i\neq j,\,\text{and }a_{i}\in\Sigma\}$, $ \sigma $ maps $ \iota $ to $ \{(c_i,\sigma(a_i))\,:\,1 \le i \le t\} $.
The strategy of covering orbits of interactions under the action of the permutation group $\Gamma$ on the symbols has been used in direct and computational methods  \cite{cck,MeagherS}, and  in randomized and derandomized methods   \cite{ColCECA}. 
The objective is to construct an array $A$ that covers all the orbits under $\Gamma$ of $t$-way interactions; to be precise, for every orbit, at least one row must cover an interaction in this orbit.  
The rows of $A$, when developed over $\Gamma$ provides an array that covers all $t$-way interactions, and therefore is a covering array. Group action here essentially works as a search space reduction technique.
In \cite{ColCECA} it is noted that using a group action appears to construct covering arrays with fewer rows than using similar methods on the covering array directly. 
We analyze the effect of group actions on the Lovász local lemma bound to further tighten the bound on the asymptotic size of covering arrays.

The action of a group $ \Gamma $ on $ \Sigma $ is  \emph{sharply transitive} if for every $ u,v \in \Sigma $ there is exactly one $ \sigma \in \Gamma $ that maps $ u $ to $ v $. When the action of $ \Gamma $ is sharply transitive on $ \Sigma $, $ |\Gamma|=|\Sigma|=v $. For example, the action of the cyclic group $ C_v $ on $ v $ symbols is sharply transitive. Similarly, the action of a group $ \Gamma $ on $ \Sigma $ is  \emph{sharply $ l $-transitive} if for all pairwise distinct $ u_1,\ldots,u_l \in \Sigma $ and pairwise distinct $ v_1,\ldots,v_l \in \Sigma $ there is exactly one $ \sigma \in \Gamma $ that maps $ u_i $ to $ v_i $ for $ 1 \le i \le l $.
\begin{thm}
\label{thm:Improvement}Let $t,\, v$ be integers with $t,v\ge2$. Then
\[
d(t,v)\le\frac{v(t-1)}{\log\left(\frac{v^{t-1}}{v^{t-1}-1}\right)}
\]
\end{thm}
\begin{proof}
Let $ \Gamma $ be a group that acts sharply transitively on $ \Sigma $. 
Let $\mathscr{C}_{t}={[k] \choose t}$, and $\tau\in\mathscr{C}_{t}$ be a collection of $t$ columns. 
The action of $\Gamma$ partitions the set of interactions involving the columns in $\tau$ into $v^{t-1}$ orbits of length $v$ each. 
We consider an $n\times k$ array $A$ with each entry  chosen independently and uniformly at random from the alphabet $\Sigma$ . 
We want to cover all the orbits for every $\tau\in\mathscr{C}_{t}$. 
The probability that there is at least one orbit involving $\tau$ that is not covered is $v^{t-1}\left(1-\frac{1}{v^{t-1}}\right)^{n}$.

For $\tau\in\mathscr{C}_{t}$ , let $A_{\tau}$ denote the event that not all the orbits involving the columns in $\tau$ are covered in $A$. 
So $\Pr[A_{\tau}]\le v^{t-1}\left(1-\frac{1}{v^{t-1}}\right)^{n}$ for all $\tau\in\mathcal{T}$. 
The event $A_{\tau}$ is not independent of  event $A_{\rho}$ if and only if $\tau$ and $\rho$ share a column. 
So $d\le{t \choose 1}{k-1 \choose t-1}<t{k \choose t-1}$.
By the Lovász local lemma, if $ev^{t-1}\left(1-\frac{1}{v^{t-1}}\right)^{n}t{k \choose t-1}<1$,
there exists an $n\times k$ array that covers every orbit on
every $t$-column combination of $A$. Solving for $n$, and
then developing $A$ over the group $\Gamma$, we obtain a covering
array of size $N$, where

\begin{eqnarray*}
N & = & vn\\
 & > & v\frac{1+\log\left(v^{t-1}t{k \choose t-1}\right)}{\log\left(\frac{v^{t-1}}{v^{t-1}-1}\right)}\\
 & \ge & v\frac{1+\log\left(v^{t-1}t\left(\frac{k}{t-1}\right)^{t-1}\right)}{\log\left(\frac{v^{t-1}}{v^{t-1}-1}\right)}\\
& = & \frac{v(t-1)\log k}{\log\left(\frac{v^{t-1}}{v^{t-1}-1}\right)}\left\{ 1+\frac{1}{(t-1)\log k}+\frac{\log v}{\log k}+\frac{\log t}{(t-1)\log k}-\frac{\log(t-1)}{\log k}\right\} \\
 & = & \frac{v(t-1)\log k}{\log\left(\frac{v^{t-1}}{v^{t-1}-1}\right)}\left\{ 1+\mbox{\rm o}(1)\right\} 
\end{eqnarray*}

This yields the required bound on $d(t,v)$.
\end{proof}

Comparing the bounds from Theorems \ref{thm:godbole} and \ref{thm:Improvement}, using the Taylor series expansion of $\log(1-x)=-x-\frac{x^{2}}{2}-O(x^{3})$, we find that
\[ \begin{array}{rcl}
\frac{t-1}{\log\left(\frac{v^{t}}{v^{t}-1}\right)}=\frac{t-1}{-\log\left(1-\frac{1}{v^{t}}\right)}&\approx&\frac{t-1}{\left(\frac{1}{v^{t}}+\frac{1}{2.v^{2t}}\right)}=\frac{v^{t}(t-1)}{1+\frac{1}{2v^{t}}}, \mbox{ and }\\
\frac{v(t-1)}{\log\left(\frac{v^{t-1}}{v^{t-1}-1}\right)}=\frac{v(t-1)}{-\log\left(1-\frac{1}{v^{t-1}}\right)}&\approx&\frac{v(t-1)}{\left(\frac{1}{v^{t-1}}+\frac{1}{2v^{2t-2}}\right)}=\frac{v^{t}(t-1)}{1+\frac{1}{2v^{t-1}}}.
\end{array} \]

Hence the bound of Theorem \ref{thm:Improvement} is tighter.
Franceti\'{c} and Stevens \cite{francetic15} also report the bound in Theorem \ref{thm:Improvement}.
Their approach uses  entropy compression arguments, and appears to be more involved than the approach here. 
Furthermore,  we can get a better improvement  by using a larger group:

\begin{thm}
\label{thm:frobenius}Let $t\ge2$ be an integer and $v$ be a prime power. Then
\[
d(t,v)\le\frac{v(v-1)(t-1)}{\log\left(\frac{v^{t-1}}{v^{t-1}-v+1}\right)}
\]
\end{thm}
\begin{proof}
Let $ \Gamma $ be a group that is sharply 2-transitive on $ v $ symbols.
Consider the action of $\Gamma$ on the set of interactions involving the columns $\tau\in{[k] \choose t}$ . 
Under the action of $\Gamma$  the $v$ interactions $\{(c_{i},v_{i})\,:\, c_{i}\in\tau,\,1\le i\le t\}$ with $v_{1}=\ldots=v_{t}$ (the \emph{constant} interactions) form a single orbit of length $v$. 
The remaining $v^{t}-v$ interactions form $\frac{v^{t-1}-1}{v-1}$ orbits, each of length $v(v-1)$. 
So the probability that a full length orbit is not covered in a $n\times k$ random array is $\left(1-\frac{v-1}{v^{t-1}}\right)^{n}$,
and the probability that at least one of these orbits is not covered in the random array is at most $\left(\frac{v^{t-1}-1}{v-1}\right)\left(1-\frac{v-1}{v^{t-1}}\right)^{n}$ by the union bound.

Using the Lovász local lemma, when $e\left(\frac{v^{t-1}-1}{v-1}\right)\left(1-\frac{v-1}{v^{t-1}}\right)^{n}t{k \choose t-1}<1$, there exists an $n\times k$ array that covers all the full orbits of interactions on all $t$-column combinations. 
Developing this array over $\Gamma$ and adding $v$ additional rows to cover the short orbit, we obtain a covering array with $N$ rows, with 

\begin{eqnarray*}
N & = & v(v-1)n+v\\
 & > & v(v-1) \frac{1+\log\left(t {k \choose t-1}\right)+\log\left(\frac{v^{t-1}-1}{v-1}\right)}{\log\left(\frac{v^{t-1}}{v^{t-1}-v+1}\right)}+v\\
 & \ge & v(v-1) \frac{1+\log\left(t \left(\frac{k}{t-1}\right)^{t-1}\right)+\log\left(\frac{v^{t-1}-1}{v-1}\right)}{\log\left(\frac{v^{t-1}}{v^{t-1}-v+1}\right)}+v\\
 & = & \frac{v(v-1)(t-1)\log k}{\log\left(\frac{v^{t-1}}{v^{t-1}-v+1}\right)}\left\{ 1+\mbox{\rm o}(1)\right\} 
\end{eqnarray*}

This proves the theorem.
\end{proof}

Again using the Taylor series expansion, the bound obtained in Theorem \ref{thm:frobenius} is tighter than the bound in Theorem \ref{thm:Improvement}, as follows.

\[
\frac{v(v-1)(t-1)}{\log\left(\frac{v^{t-1}}{v^{t-1}-v+1}\right)}=\frac{v(v-1)(t-1)}{-\log\left(1-\frac{v-1}{v^{t-1}}\right)}\approx\frac{v(v-1)(t-1)}{\left\{ \frac{v-1}{v^{t-1}}+\frac{(v-1)^{2}}{2v^{2t-2}}\right\} }=\frac{v^{t}(t-1)}{1+\frac{v-1}{2v^{t-1}}}
\]

Let $G$ be the Frobenius group defined on the finite field $\mathbb{F}_{v}$, i.e. $G=\{g:\mathbb{F}_{v}\rightarrow\mathbb{F}_{v}\,:\, g(x)=ax+b,\, x,a,b\in\mathbb{F}_{v},\, a\neq0\}$. $ G $ is an efficiently constructible group that acts sharply 2-transitively on the set of $ v $ symbols and can be used for practical construction of covering arrays \cite{ColCECA}.

It is natural to consider the action of larger groups in seeking further improvements.  
One simple but important idea in Theorem \ref{thm:frobenius} is to treat full length orbits using the Lovász local lemma, adjoining a small number of additional rows to cover the short orbits. 
Thus far we have treated a sharply 1-transitive group (the cyclic group) and a sharply 2-transitive group (the Frobenius group).
In order to generalize,  the next natural choice is the projective general linear (PGL) group for $v=q+1$ where $q$ is a prime power, which is a sharply 3-transitive group of order $v(v-1)(v-2)$. Let $\Gamma$ be the PGL group on $v$ symbols.
The action of $\Gamma$ on $t$-way interactions forms orbits of lengths $v$, $v(v-1)$, and $v(v-1)(v-2)$.  
Constant interactions lie in orbits of length $v$, interactions involving precisely two distinct symbols lie in orbits of length $v(v-1)$, and the $r=\frac{v^{t-1}-(v-1)(2^{t-1}-1)-1}{(v-1)(v-2)}$ others lie in full length orbits.  
Constant orbits can be handled as in Theorem \ref{thm:frobenius}, and full length orbits can be treated using the Lov\'asz local lemma. Unlike constant orbits, orbits of length $v(v-1)$ cannot be covered with a  number of rows that is independent of $k$. 
If we  cover the orbits of length $v(v-1)$ as we covered full length orbits, we see no improvement over Theorem \ref{thm:frobenius}.
We adapt a method from \cite{CCL} to gain occasional improvements.

\begin{thm}
\label{thm:pgl}Let $t\ge2$ be an integer and $v-1$ be a prime power. Then

\[
d(t,v)\le\frac{v(v-1)(v-2)(t-1)}{\log\left\{ \frac{v^{t-1}}{v^{t-1}-(v-1)(v-2)}\right\} }+\frac{v(v-1)(t-1)}{\log\left(\frac{2^{t-1}}{2^{t-1}-1}\right)}
\]
\end{thm}
\begin{proof}
Let $\Gamma$ be the PGL group acting on $v$ symbols.

\emph{Covering orbits of length $v(v-1)(v-2)$:} The probability that
at least one orbit of length $v(v-1)(v-2)$ is not covered in an array with $n$ rows is $p\le r   \left(1-\frac{(v-1)(v-2)}{v^{t-1}}\right)^{n}$.
As before, $d<t{k \choose t-1}$. Using the  Lovász local lemma, if $ep(d+1)\le1$ there
is an array with $n$ rows that covers all orbits of length $v(v-1)(v-2)$.
Developing over $\Gamma$ we obtain an array of size

\begin{align*}
    & v(v-1)(v-2) \frac{1+\log\left\{ t {k \choose t-1}\right\} +\log r}{\log\left\{ \frac{v^{t-1}}{v^{t-1}-(v-1)(v-2)}\right\} }\\
    = & v(v-1)(v-2) \frac{1+(t-1) \log k+\log t+\log r}{\log\left\{ \frac{v^{t-1}}{v^{t-1}-(v-1)(v-2)}\right\} }\\
    = & \frac{v(v-1)(v-2)(t-1)}{\log\left\{ \frac{v^{t-1}}{v^{t-1}-(v-1)(v-2)}\right\} }\log k \left\{ 1+o(1)\right\} 
\end{align*}

Using the Taylor series expansion:

\[
\frac{v(v-1)(v-2)(t-1)}{\log\left\{ \frac{v^{t-1}}{v^{t-1}-(v-1)(v-2)}\right\} }\approx\frac{v^{t}(t-1)}{1+\frac{(v-1)(v-2)}{2  v^{t-1}}}
\]

\emph{Covering orbits of length $v(v-1)$:} Use a binary covering array  on every pair of symbols, adding ${v \choose 2} \CAN(t,k,2)$ rows to cover all interactions in orbits of length $v(v-1)$.
Applying Theorem \ref{thm:Improvement} to bound $\CAN(t,k,2)$, in this way we add \[ {v \choose 2} \frac{2(t-1)}{\log\left(\frac{2^{t-1}}{2^{t-1}-1}\right)}\log k \left\{ 1+o(1)\right\} =\frac{v(v-1)(t-1)}{\log\left(\frac{2^{t-1}}{2^{t-1}-1}\right)}\log k \left\{ 1+o(1)\right\}\] rows.

So $d(t,v)\le\frac{v(v-1)(v-2)(t-1)}{\log\left\{ \frac{v^{t-1}}{v^{t-1}-(v-1)(v-2)}\right\} }+\frac{v(v-1)(t-1)}{\log\left(\frac{2^{t-1}}{2^{t-1}-1}\right)}$.
\end{proof}

In the bound of Theorem \ref{thm:pgl}, the first term dominates the second. 
However, only when $t\in\{3,4\}$ and $v$ is sufficiently large does Theorem \ref{thm:pgl} give a tighter bound on $d(t,v)$ than that given by Theorem \ref{thm:frobenius}. 
Moreover, Theorem \ref{thm:frobenius} gives a tighter bound in many situations; when $t=5$, it is tighter when $v\le29$, and for larger $t$ the values of $v$ for which  it is tighter extend further. 
Hence the natural avenue of generalization to larger groups does not appear to be fruitful.

So far in our discussion of group action we have emphasized only the search space reduction aspect. Now we mention a side benefit inherent to sharply transitive group actions that further validates their role. By using sharply transitive (or sharply $l$-transitive) group actions we can further reduce the dependence between different bad events. Concretely, consider the cyclic group used in Theorem \ref{thm:Improvement}. For any set of $t$ columns $\tau$, if we fix the symbols in a specific column $c$ and select symbols in the remaining $t-1$ columns independently and uniformly at random, the probability of a ``bad event'' (i.e. at least one orbit not being covered) remains unchanged. This suggests that all the ``bad events'' on the $t$-set of columns that share only the column $c$ with $\tau$ are mutually independent of the ``bad event'' on $\tau$. Therefore, we can set $ d \le t {k - 1 \choose t-1} - {k-t \choose t-1}$. Although this improved estimate does not change the asymptotic bound in Theorem \ref{thm:Improvement}, in some cases it reduces the actual number of rows required in practice \cite{SarkarColbourn2}. Similar reduction in dependence may be obtained when we apply the Lovász local lemma to cover the full length orbits under sharply $l$-transitive group actions in Theorem \ref{thm:frobenius} and Theorem \ref{thm:pgl}. 

Although the proofs of Theorems \ref{thm:Improvement} and \ref{thm:frobenius} are non-constructive, construction algorithms can be obtained realizing the same bounds.
In some remarkable work, Moser et al. \cite{moser09,moser10} provide a constructive version of the Lovász local lemma.
Applying the method of  \cite{moser10}  to  covering array construction  provides a randomized algorithm:
\begin{enumerate}
\item For group $\Gamma$ acting on symbols $\Sigma$, determine the smallest value of $n$ by applying Theorems \ref{thm:Improvement} or \ref{thm:frobenius}.
\item Construct an $n\times k$ array with each entry  chosen independently and uniformly at random from $\Sigma$. 
\item Check sequentially in some fixed order that each of the (full length) orbits is covered in the array. If all are covered, report success and stop.
\item For the first  orbit that is not covered, ``re-sample'' each  column in that orbit, by choosing new entries in the column independently and uniformly at random from $\Sigma$. 
\item Restart the check from the beginning. 
\end{enumerate}
The expected number of column re-samplings is polynomially bounded \cite{moser10}.

The storage requirements are quite modest; in order to determine whether resampling is necessary, one maintains a single list indexed by the orbits of $\Sigma^t$.  
A set of $t$ columns can be treated without regard to the coverage in other sets of $t$ columns. See Sarkar and Colbourn \cite{SarkarColbourn2} for a detailed exploration of this algorithm for the practical purpose of covering array construction.

\section{Partial array construction with the Lovász local lemma}\label{sec:Partial-array-LLL}

Can alteration  techniques such as those in Section \ref{sub:Partial-array} be applied in conjunction with the techniques in Section \ref{sec:Partial-dependence}?
More precisely, can we use the Lovász local lemma to obtain a suitable partial covering array that covers ``most'' interactions with fewer rows than a random array? 
We make some first steps in addressing this question. 

To provide a better appreciation of our strategy, we start with an alternative proof of Theorem \ref{thm:two-stage}.
Let $X$ be a subset of interactions, and let $x=\frac{|X|}{{k \choose t}v^{t}}$.
Using the  union bound, the number of rows in a random array that covers all  interactions in $X$ is expected to be $\frac{\log|X|}{\log\left(\frac{v^{t}}{v^{t}-1}\right)}=\frac{\log\left(x {k \choose t} v^{t}\right)}{\log\left(\frac{v^{t}}{v^{t}-1}\right)}$.
The expected number of uncovered interactions in a random array with
$n>\frac{\log|X|}{\log\left(\frac{v^{t}}{v^{t}-1}\right)}$ rows is
$R={k \choose t}v^{t}\left(1-\frac{1}{v^{t}}\right)^{n}<{k \choose t}v^{t}\left(1-\frac{1}{v^{t}}\right)^{\frac{\log|X|}{\log\left(\frac{v^{t}}{v^{t}-1}\right)}}=\frac{{k \choose t}v^{t}}{|X|}=\frac{1}{x}$.
So there is a partial covering array with $\frac{\log x+\log{k \choose t}+t\log v}{\log\left(\frac{v^{t}}{v^{t}-1}\right)}$ rows that covers all the interactions in $X$, and has at most $1/x$ uncovered interactions. 
Adding one row to cover each uncovered interaction we obtain a covering array with $\frac{1}{x}+\frac{\log x+\log{k \choose t}+t\log v}{\log\left(\frac{v^{t}}{v^{t}-1}\right)}$ rows. 
Applying elementary calculus, when $x=\log\left(\frac{v^{t}}{v^{t}-1}\right)$ the number of rows in the covering array is the minimum, $\frac{\log{k \choose t}+t\log v+\log\log\left(\frac{v^{t}}{v^{t}-1}\right)+1}{\log\left(\frac{v^{t}}{v^{t}-1}\right)}$. 
This is the same as the bound in Theorem \ref{thm:two-stage}. 
Applying the Taylor series expansion of $\log(1-x)$, it can be shown that $\log\left(\frac{v^{t}}{v^{t}-1}\right)\approx\frac{1}{v^{t}}$.
So in aiming to cover $|X|={k \choose t}v^{t}x={k \choose t}v^{t}\log\left(\frac{v^{t}}{v^{t}-1}\right)\approx{k \choose t}$ interactions with a random array, we  cover almost ${k \choose t}v^{t}-v^{t}$ interactions.

Now we consider a variation. 
Start with a target set of interactions $X$.
Cover all interactions in $X$ using an array $A$ with $n$ rows, produced by the randomized algorithm of \cite{moser10}. 
Then $A$ may also cover some  interactions not in $X$, but does not in general cover all interactions. 
To finish,  cover the interactions that still remain in the second stage. 

To analyze the effectiveness, we need an upper bound on the probability that an interaction (not in $X$) is not  covered given that all the interactions in $X$ have been covered. 
We describe how to estimate this probability.
Haeupler et al. \cite{haeupler10} introduce the  conditional Lovász local lemma distribution. 
Let $\mathscr{X}=\{X_{1},X_{2},\ldots,X_{n}\}$ be a set of $n$ independent random variables. 
Let $\mathscr{A}=\{A_{1},A_{2},\ldots,A_{m}\}$ be a set of $m$ events that are determined by the random variables in $\mathscr{X}$. 
Let $vbl(A_{i})\subseteq\mathscr{X}$ be the minimal set of random variables that determine the event $A_{i}$. 
Let $B\notin\mathscr{A}$ be another event determined by some subset of random variables in $\mathscr{X}$. 
For any event $A\in\mathscr{A}\cup\{B\}$, let $\Gamma(A)$ be the set of other events $A'$ in $\mathscr{A}$ such that $vbl(A)\cap vbl(A')\neq\emptyset$.
Let $x:\mathscr{A}\rightarrow[0,1)$ such that for all $A\in\mathscr{A}$, $\Pr[A]\le x(A)\prod_{A'\in\Gamma(A)}(1-x(A'))$. 
The conditional Lovász local lemma distribution is the probability distribution over the random variables in $\mathscr{X}$, given that all the events in $\mathscr{A}$ are avoided. 
The probability of the event $B$ is given by:

\[
\Pr[B|\wedge_{i=1}^{m}\bar{A_{i}}]\le\frac{\Pr[B]}{\prod_{A\in\Gamma(B)}(1-x(A))},
\]
\noindent
where $\Pr[B]$ is the unconditional probability of the event $B$ \cite{haeupler10}.
We exploit the conditional Lovász local lemma distribution to analyze our two stage strategy. 
Let $R$ be a set of ${k \choose t}$ interactions such that for every $t$-column combination $\tau$ there is an interaction in $R$ involving all the columns in $\tau$. 
Let $A_{i}$ be the event that the $i$th interaction in $R$ is not covered in an i.i.d. random $n\times k$ array. 
Each $A_{i}$ is dependent on $d\le{t \choose 1}{k \choose t-1}$ other such events. 
Let $p$ be the probability of the event $A_{i}$.
Then $p=\left(1-\frac{1}{v^{t}}\right)^{n}$. 
Following the proof of the symmetric version of Lovász local lemma,  set $x(A_{i})=\frac{1}{d+1}$, and note that $\left(1-\frac{1}{d+1}\right)^{d}>\frac{1}{e}$. 
If $ep(d+1)\le1$, then $p\le\frac{1}{e(d+1)}<\frac{1}{(d+1)}\left(1-\frac{1}{d+1}\right)^{d}$, i.e. there is an $n\times k$ array that covers all  interactions in $R$.
Hence for $n\ge\frac{\log\left\{ et{k \choose t-1}\right\} }{\log\left(\frac{v^{t}}{v^{t}-1}\right)}$, there exist $n\times k$ arrays that cover all  interactions in $R$.
Let $\iota$ be an interaction that is not in $R$. 
The event that $\iota$ is not covered is dependent on at most $d$ events $A_{i}$.
Under the conditional Lovász local lemma distribution, the probability that $\iota$ is not covered is $p'\le\frac{p}{\left(1-\frac{1}{d+1}\right)^{d}}<ep=e\left(1-\frac{1}{v^{t}}\right)^{n}$.
So the expected number of uncovered interactions in the array is at most 
\[
{k \choose t}(v^{t}-1)e\left(1-\frac{1}{v^{t}}\right)^{n}\le\frac{{k \choose t}(v^{t}-1)e}{e t {k \choose t-1}}\le\frac{\left(\frac{ke}{t}\right)^{t}(v^{t}-1)}{t \left(\frac{k}{t-1}\right)^{t-1}}=k \frac{e^{t}(v^{t}-1)}{t^{2}} \left(1-\frac{1}{t}\right)^{t-1}
\]

By finding an $n\times k$ array with at most $\lfloor k \frac{e^{t}(v^{t}-1)}{t^{2}} \left(1-\frac{1}{t}\right)^{t-1}\rfloor$ uncovered interactions and then adding one extra row for each uncovered interaction, we can construct a covering array with $N=\frac{\log\left\{ e t {k \choose t-1}\right\} }{\log\left(\frac{v^{t}}{v^{t}-1}\right)}+\lfloor k \frac{e^{t}(v^{t}-1)}{t^{2}} \left(1-\frac{1}{t}\right)^{t-1}\rfloor$ rows. 
Unfortunately, the bound on $N$ is linear in $k$, and so is ineffective when $k$ is large. 
However, as Figure \ref{fig:Comp-lll-2-stage} shows, before the linearity in $k$ dominates, this bound improves substantially on a direct application of the Lovász local lemma.  
Indeed the utility of the bound lies in its ability to address situations in which $k$ is of ``intermediate'' size.

To avoid the linearity in $k$, we can employ an improved second stage.
In practice, we could apply the density algorithm; to obtain a general bound we employ the discrete Stein-Lovász-Johnson bound.
Figure \ref{fig:Comp-lll-2-stage} compares the Lovász local lemma  bound, the simple two-stage bound,  and the density two-stage bound. 
The application of the density algorithm in the second stage reduces the number of required rows to logarithmic in $k$.

\begin{figure}
\begin{centering}
\includegraphics[scale=0.7]{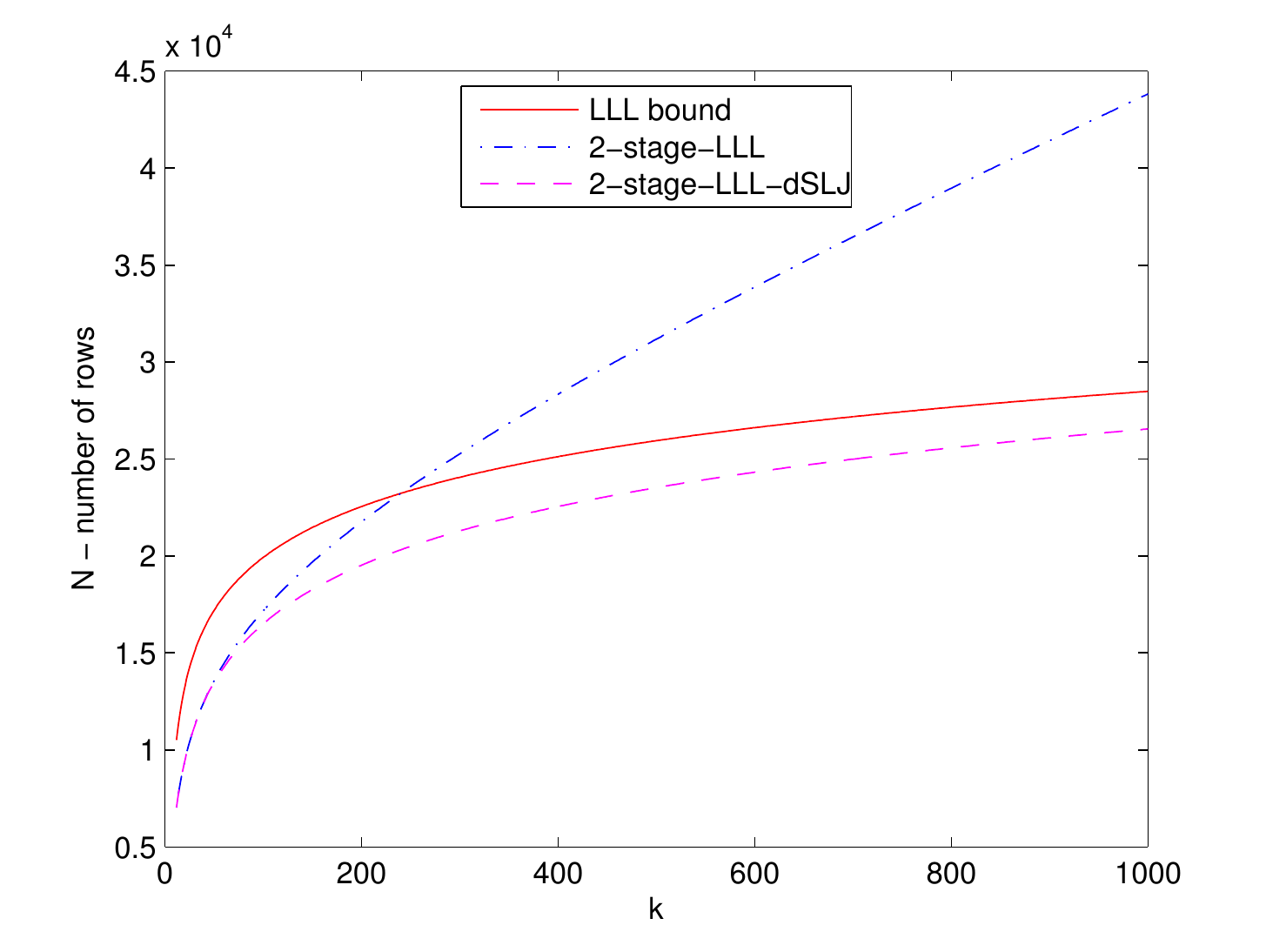}
\par\end{centering}

\caption{Comparison of covering array size bounds for the two-stage algorithm, the two-stage algorithm with density (discrete SLJ) in the second stage, and the Lovász local lemma for $t=6$ and $v=3$. 
Up to $k\sim200$, the two-stage algorithm outperforms the Lovász local lemma bound.
Application of the density algorithm in the second stage results in improvement  for a higher range of $k$ values. \label{fig:Comp-lll-2-stage}}
\end{figure}

\section{Conclusion}\label{sec:Conclusion}
The Stein-Lovász-Johnson and the Lovász local lemma methods for obtaining asymptotic bounds on $\CAN(t,k,v)$  also yield efficient construction techniques.  
Exploiting 2-transitive group actions, we have shown that the Lovász local lemma can be applied to obtain an upper bound on covering array numbers that improves upon all known bounds.  
In addition, by examining group action and by considering two-stage methods, we have developed upper bounds that are tighter when the number of factors is of intermediate size.
Each of the bounds obtained yields an efficient construction procedure and can be easily computed.  
Earlier density methods are in principle efficient, but suffer from  challenging storage requirements to maintain a list of $\binom{k}{t} v^t$ $t$-way interactions. 
The two-stage methods developed here obviate the need for such extensive tables, and hence provide construction algorithms of practical importance; see \cite{SarkarColbourn2}.

Our two-stage method based on the Lovász local lemma would be improved if a better upper bound than discrete-SLJ were known on the number of bad events when $n$ rows are selected.
This appears not to be straightforward, but is certainly of potential value.  
Another direction of interest is to explore different techniques for sample space reduction than the transitive group actions considered here; our results indicate that such sample space reduction can provide substantial improvements in the bounds on   $\CAN(t,k,v)$.

\section*{Acknowledgments}

The  research was supported in part by the National Science Foundation under Grant No. 1421058.

\bibliographystyle{abbrv}
\bibliography{CoveringArray}

\end{document}